\documentclass[11pt]{article}
\usepackage{amsmath,amssymb}

\newtheorem{propo}{{\bf Proposition}}[section]
\newtheorem{coro}[propo]{{\bf Corollary}}
\newtheorem{lemma}[propo]{{\bf Lemma}} \newtheorem{theor}[propo]{{\bf
Theorem}} \newtheorem{ex}{{\sc Example}}[section]

\newenvironment{proof}{{\bf Proof.}}{$\Box$}

\def\C{{\mathbb C}}
\def\N{{\mathbb N}}
\begin{document}
\vspace*{1.0in}

\begin{center} ABELIAN SUBALGEBRAS AND IDEALS OF MAXIMAL DIMENSION IN SOLVABLE LEIBNIZ ALGEBRAS

\end{center}
\bigskip

\centerline {Manuel Ceballos } \centerline {Dpto. de Ingenier\'{\i}a, Universidad Loyola Andaluc\'{\i}a} \centerline
{Av. de las Universidades, s/n, 41704 Dos Hermanas, Sevilla, Spain.}
\centerline {and}

\centerline {David A. Towers} \centerline {Department of
Mathematics, Lancaster University} \centerline {Lancaster LA1 4YF,
England}
\bigskip

\begin{abstract}
In this paper, we compare the abelian subalgebras and ideals of maximal dimension for finite-dimensional Leibniz algebras. We study Leibniz algebras containing abelian subalgebras of codimension $1$, solvable and supersolvable Leibniz algebras for codimensions $1$ and $2$, nilpotent Leibniz algebras in case of codimension $2$ and we also analyze the case of $k$-abelian $p$-filiform Leibniz algebras. Throughout the paper, we also give examples to clarify some results and the need for restrictions on the underlying field.

\end{abstract}


\section{Introduction}
\medskip

Leibniz algebras were introduced by Loday \cite{Loday} in 1993. They are a particular case of non-associative algebras and a non-anticommutative generalization of Lie algebras. In fact, Leibniz algebras inherit an important property of Lie algebras: the right-multiplication operator is a derivation. Many well-known results on Lie algebras can be extended to Leibniz algebras. There exists an extensive body of research on these algebras due to its own theoretical importance as well as its application to other fields such as Engineering, Physics and Applied Mathematics (see \cite{ao,Bosco,CS,Fila} for example). However, some aspects about Leibniz algebras remain unknown such as their classification. According to \cite{barnes}, this classification can be reduced to semisimple and solvable algebras. In 1890, Killing and Cartan classified semisimple Lie algebras and the Leibniz case can be described from simple Lie ideals. The current techniques do not allow researchers to deal successfully with the classification problem of solvable Lebiniz algebras. Therefore, the need for studying different properties of Leibniz algebras arises. For example, conditions on the lattice of subalgebras of a Leibniz algebra could lead to information about the algebra itself. Studying abelian subalgebras and ideals of Leibniz algebras constitutes the main goal of this paper.

Throughout $L$ will denote a finite-dimensional Leibniz algebra over a field $F$ and for every $x,y,z \in L$, the so-called Leibniz identity is satisfied:
\[  [x,[y,z]]=[[x,y],z]-[[x,z],y]
\]
In other words, the right multiplication operator $R_x : L \rightarrow L : y\mapsto [y,x]$ is a derivation of $L$. As a result, such algebras are sometimes called {\it right} Leibniz algebras, and there is a corresponding notion of {\it left} Leibniz algebra. Clearly, the opposite of a right Leibniz algebra is a left Leibniz algebra so, for our purposes, it does not matter which is used. Every Lie algebra is a Leibniz algebra and every Leibniz algebra satisfying $[x,x]=0$ for every element is a Lie algebra. The set of right multiplication operators $\mathfrak{R}(L)=\{R_x : x\in L\}$ forms an ideal of the Lie algebra of derivations of $L$, Der($L$).

Let us consider $I=span\{x^2:x\in L\}$. Then $I$ is an ideal of $L$ and $L/I$ is a Lie algebra called the {\em liesation} of $L$.
We shall call $L$ {\em supersolvable} if there is a chain $0=L_0 \subset L_1 \subset \ldots \subset L_{n-1} \subset L_n=L$, where $L_i$ is an $i$-dimensional ideal of $L$.  We define the following series:
\[ L^1=L,L^{k+1}=[L^k,L] \hbox{ and } L^{(1)}=L,L^{(k+1)}=[L^{(k)},L^{(k)}] \hbox{ for all } k=2,3, \ldots
\]
Then $L$ is {\em nilpotent} (resp. {\em solvable}) if $L^n=0$ (resp. $ L^{(n)}=0$) for some $n \in \N$. The {\em nilradical}, $N(L)$, (resp. {\em radical}, $R(L)$) is the largest nilpotent (resp. solvable) ideal of $L$. The {\it Frattini ideal} of $L$, $\phi(L)$, is the largest ideal of $L$ contained in all maximal subalgebras of $L$. We will denote the {\it centre} of $L$ by $Z(L)=\{x \in L: [x,y]=[y,x]=0, \, \forall \, y \in L\}$. If $H,K$ are subsets of $L$ with $H\subseteq K$, the {\em right centraliser of $H$ in $K$} is $C_K^r(H)=\{k\in K : [H,k]=0\}$.
\par
A nilpotent Leibniz algebra $L$ of dimension $n$ is said to be $p$-{\em filiform} if
${\rm dim}(L^i)=n-p-i+1$, for $2 \leq i \leq n-p+1$. We say that $L$ is $k$-{\em abelian} if $k$ is the smallest positive integer such that $L^k$ is abelian. For general properties of Leibniz algebras see \cite{book}.

Algebra direct sums will be denoted by $\oplus$, whereas vector space direct
sums will be denoted by $\dot{+}$. The assumptions on the field $F$ will be specified in each result. We consider the following invariants of $L$:
$$\alpha(L) = \max \{\, \dim (A) \, | \, A \,\, {\rm is \,\, an \,\, abelian \,\, subalgebra \,\,
of \,\, } L\},$$ $$\beta(L)  = \max \{\,
\dim (B) \, | \, B \,\, {\rm is \,\, an \,\,
abelian \,\, ideal \,\,of \,\,}L\}.$$
Both invariants are important for many reasons. For example, they are
very useful for the study of contractions
and degenerations; in particular, an algebra $L_1$ does not degenerate into $L_2$ if $\dim \alpha (L_1)> \dim \alpha(L_2)$ (see \cite[Corollary (6)]{R}. There is a large literature, in particular, for
low-dimensional Lie algebras, see
\cite{GRH,BU10,NPO,SEE,GOR}, and, more recently, for Leibniz algebras, see \cite{CKLO, FKRV, IKV, MHY,R, RA} and the references given therein.
\par

The authors of this paper have already studied these invariants for Lie algebras in
\cite{bc,CT,Tow}. More concretely, in \cite{bc} it was shown that for a solvable Lie algebra $L$ over an algebraically closed field of characteristic zero, $\alpha(L)=\beta(L)$ and the cases of abelian subalgebras of codimension $1$ and $2$ were also studied. In \cite{CT}, the authors proved that $n$-dimensional supersolvable Lie algebras $L$ with $\alpha(L)=n-2$ verify $\beta(L)=n-2$. They also proved the same for nilpotent Lie algebras over a field of characteristic different from two having abelian subalgebras of codimension $3$. This result was generalised for supersolvable Lie algebras in \cite{Tow} and the same was proved for nilpotent Lie algebras containing abelian subalgebras of codimension $4$. However, we have not found a similar study in the literature for Leibniz algebras. Therefore, studying abelian subalgebras and ideals of maximal dimension in Leibniz algebras constitutes the main goal of this paper.

The structure of this current paper is as follows. In Section $2$, we study the case of codimension one proving that every $n$-dimensional solvable Leibniz algebra $L$ over a field of characteristic different from two with $\alpha(L)=n-1$ also satisfies $\beta(L)=n-1$. We also show that this is true for supersolvable Leibniz algebra over any field. In Section $3$, we prove that the alpha and beta invariants also coincide for nilpotent Leibniz algebras over a field of characteristic different from two. Moreover, we characterise solvable Leibniz algebras containing abelian subalgebras of codimension two. In this case, the equality of alpha and beta can be assured for supersolvable Leibniz algebras. Finally, we study the case of filiform Leibniz algebras proving that there is a unique abelian ideal of maximal dimension for $k$-abelian $p$-filiform Leibniz algebras. We also provide examples in each section.

\section{Abelian subalgebras of codimension one}
\begin{theor}\label{solvcodim1} Let $L$ be an $n$-dimensional Leibniz algebra over a field of characteristic different from $2$ satisfying $\alpha(L)=n-1$. Then $\beta(L)=n-1$.
\end{theor}
\begin{proof} Let $A$ be an abelian subalgebra of dimension $n-1$. If $L^2\subseteq A$ then $A$ is also an abelian ideal. Otherwise, choose a basis $e_1,\dots, e_n$ for $L$ such that $A=$ span$\{e_2,\ldots, e_n\}$. Then
\[  [e_1,e_j]=\sum_{i=1}^n \alpha_{ji}e_i,   [e_j,e_1]=\sum_{i=1}^n \beta_{ji}e_i \hbox{ and } e_1^2=\sum_{i=1}^n \gamma_ie_i
\] where $\alpha_{ji}, \beta_{ji} \in F$ for $j=2,\dots, n$. Then $e_1^2, [e_1,e_j]+[e_j,e_1]\in I$ , so for $j,k=2,\ldots n$
\begin{align}
0 & =[e_j,e_1^2] = \gamma_1[e_j,e_1] \\
0 & =[e_j,[e_1,e_k]+[e_k,e_1]] = (\alpha_{k1}+\beta_{k1})[e_j,e_1] \\
 0 & =[e_1,[e_2,e_j]]  =[[e_1,e_2],e_j]-[[e_1,e_j],e_2] \nonumber \\
&  =\alpha_{21}[e_1,e_j]-\alpha_{j1}[e_1,e_2] \\
0 & =[[e_2,e_j],e_1]=[e_2.[e_j,e_1]]+[[e_2,e_1],e_j] \nonumber \\
& =\beta_{j1}[e_2,e_1]+\beta_{21}[e_1,e_j].
\end{align}
Now (1) and (2) imply that $[e_j,e_1]=0$ for all $j=2,\ldots, n$, or $\gamma_1=0$ and $\beta_{j1}=-\alpha_{j1}$ for all $j=2,\ldots, n$.
\par

Suppose first that the former holds. If $A$ is not an ideal, there is an $e_j$ such that $[e_1,e_j]\notin A$. By relabelling we can assume that $[e_1,e_2]\notin A$, whence $\alpha_{21}\neq 0$. Put $v_j=\alpha_{j1}\alpha_{21}^{-1}e_2-e_j$ for $j=2,\ldots, n$. Then $[e_1,e_2]\in I$ so $[e_1,e_2]^2=0$, $[v_j,[e_1,e_2]]=0$ and
\[ [[e_1,e_2],v_j]=\alpha_{j1}[e_1,e_2]-\alpha_{21}[e_1,e_j] = 0 \hbox{ by (3)},
\]
so $B=F[e_1,e_2]+Fv_3+\ldots +Fv_n$ is an abelian subalgebra of $L$. Also,
\begin{align*}
[e_1,v_j] & = \alpha_{j1}\alpha_{21}^{-1}[e_1,e_2]-[e_1,e_j]=0, & [v_j,e_1]=0 \\
[[e_1,e_2],e_1] & =[e_1,[e_2,e_1]]+[e_1^2,e_2]=\gamma_1[e_1,e_2] \in B, & [e_1,[e_1,e_2]]=0,
\end{align*} and $B$ is an ideal, as required.
\par

So suppose now that the latter holds. If  $\alpha_{j1}=0=\beta_{j1}$ for all $j=2,\dots, n$, we have that $A$ is an ideal of $L$. Hence, we may again assume that $\alpha_{21}\not = 0$. Now (3) and (4) imply that
\[ \alpha_{21}[e_1,e_j]=\alpha_{j1}[e_1,e_2]=-\beta_{j1}[e_1,e_2]=\beta_{21}[e_j,e_1]=-\alpha_{21}[e_j,e_1],
\] so $[e_1,e_j]=-[e_j,e_1]$ for all $j=2,\ldots, n$. Moreover,
\begin{align*}
 \alpha_{21}e_1^2+\sum_{i=2}^n\alpha_{2i}[e_1,e_i]&=[e_1,[e_1,e_2]]=[e_1^2,e_2]-[[e_1,e_2],e_1] \\
& =-\alpha_{21}e_1^2-\sum_{i=2}^n\alpha_{2i}[e_i,e_1] \\
& =-\alpha_{21}e_1^2+\sum_{i=2}^n\alpha_{2i}[e_1,e_i].
\end{align*} It follows that $e_1^2=0$, whence $L$ is a Lie algebra and the result follows from \cite[Proposition 3.1]{bc}
\end{proof}
\medskip

The above result does not hold over a field of characteristic $2$, as the following example shows.

\begin{ex} Let $L=Fe_1+Fe_2+Fe_3$ where $[e_1,e_2]=[e_2,e_1]=e_1$, $e_1^2=e_3$ and all other products zero. This is a Leibniz algebra when $F$ has characteristic $2$. Then $A=Fe_2+Fe_3$ is an abelian subalgebra of dimension $2$. However, any ideal of dimension $2$ must contain $L^2=Fe_1+Fe_3$ and this is not abelian.
\end{ex}

\begin{propo} Let $L$ be a non-abelian Leibniz algebra over an algebraically closed field. Then $L$ has a maximal subalgebra $A$ which is abelian if and only if $A$ has codimension one in $L$.
\end{propo}
\begin{proof} Suppose first that $L$ has a maximal subalgebra $A$ which abelian. If $A$ is an ideal of $L$ the result is proved. So suppose that $L$ is self-idealising, in which case it is a Cartan subalgebra of $L$ Then the Fitting decomposition of $L$ relative to $A$ is $L=A\dot{+}L_1$, where $L_1=\cap_{i=1}^{\infty}R_A^i(L)$, by \cite[Proposition 2.11]{book}. Then $\{R_a |_{L_1} : a\in A\}$ is a set of simultaneously triangulable linear maps. Hence there exists $b\in B$ such that $[b,a]=R_a(b)=\lambda_ab$ for all $a\in A$, where $\lambda_a\in F$. Put $B=A+Fb$.
\par

Suppose first that $I\subseteq A$. Then $[a,b]+\lambda_ab=[a,b]+[b,a]\in A$, so $[a,b]\in B$. Also, $b^2\in I\subseteq A$. It follows that $B$ is a subalgebra of $A$, and so $L=B=A+Fb$, as required.
\par

So suppose now that $I\not\subseteq A$. Then $L=A+I$ and $L_1\subseteq L^2\subseteq I$. Hence $[a,b]=0=b^2$, from which we have that $B$ is a subalgebra again, and $B=L$
\par

The converse is clear.
\end{proof}
\medskip

If $B$ is an ideal of $L$ and $A/B$ is a minimal ideal of $L/B$ we call $A/B$ a {\em chief factor} of $L$.

\begin{theor}\label{solv} Let $L$ be a solvable Leibniz algebra over any field $F$. Then $L$ has a maximal subalgebra that is abelian if and only if one of the following occurs:
\begin{itemize}
\item[(i)] $L$ has an abelian ideal of codimension $1$ in $L$;
\item[(ii)] $L$ splits over $L^2$ and either
\begin{itemize}
\item[(a)] $L^2=I$, $L^{(2)}=\phi(L)=Z(L)\cap L^2=0$ and $L^2$ is a minimal ideal of $L$; or
\item[(b)] $L^{(2)}+I=\phi(L)= Z(L)\cap L^2$ and $L^2/(L^{(2)}+I)$ is a chief factor of $L$.
\end{itemize}
\end{itemize}
\end{theor}
\begin{proof} Suppose that $L$ has a maximal subalgebra $M$ that is abelian, and that case (i) does not hold. Then, by the maximality, $M$ is self-idealising and so is a Cartan subalgebra of $L$. Also, by the solvability, there is a $k\geq 1$ such that $L^{(k)}\not\subseteq M$ but $L^{(k+1)}\subseteq M$. Then $L=M+L^{(k)}$, whence $L^2\subseteq L^{(k)}$. It follows that $L^{(2)}\subseteq L^{(k+1)}\subseteq M$.
\par

Suppose first that $I\not\subseteq M$. Then $L=M+I$ and $L^2=I$, so $L^{(2)}=0$. Then $M$ is a Cartan subalgebra of $L$ and $L$ has Fitting decomposition $L=M\dot{+}L_1$ relative to $\mathfrak{R}(M)$ (see \cite[Proposition 2.1]{book}). Then $L_1\subseteq L^2 =I$ and so is an ideal of $L$. Moreover, since $M$ is a maximal subalgebra of $L$, $L_1$ is a minimal ideal of $L$ and $L_1=L^2$. Clearly, $\phi(L)=0$ and $Z(L)\cap L^2=0$, since, otherwise, $L^2\subseteq Z(L)$, by the minimality. We thus have case (ii)(a).
\par

So suppose now that $I\subseteq M$. If $S$ is a subalgebra of $L$, denote by $\overline{S}$ its image under the canonical homomorphism onto $L/(L^{(2)}+I)$. Then $\overline{M}$ is a Cartan subalgebra of $\overline{L}$, by \cite[Corollary 6.3]{barnes} and $\overline{L}$ has Fitting decomposition $\overline{L}=\overline{M}\dot{+}\overline{L}_1$ relative to $\mathfrak{R}(\overline{M})$ (see \cite[Proposition 2.1]{book}). Then $\overline{L}_1\subseteq \overline{L}^2 =\overline{L^2}$ and so is a right ideal of $\overline{L}$. But $\overline{L}$ is a Lie algebra, so $\overline{L}_1$ is an ideal of $\overline{L}$. Moreover, since $M$ is a maximal subalgebra of $L$, $\overline{L}_1$ is a minimal ideal of $\overline{L}$ and $\overline{L}_1=\overline{L^2}$. It follows that $L^2/(L^{(2)}+I)$ is a chief factor of $L$. Clearly $\phi(\overline{L})=0$, so $\phi(L)\subseteq L^{(2)}+I$.
\par

Now $L^{(2)}+I\subseteq M$, so $M\subseteq C_L(L^{(2)}+I)$. But $C_L(L^{(2)}+I)$ is an ideal of $L$ and $M$ is self-idealising, so $L=C_L(L^{(2)}+I)$ and $L^{(2)}+I\subseteq Z(L)\cap L^2\subseteq L^2$. If $Z(L)\cap L^2=L^2$ then $L^2\subseteq Z(L)$ which implies that $L$ is nilpotent and so $M$ is an ideal of $L$, a contradiction. Hence $L^{(2)}+I=Z(L)\cap L^2$.
\par

 If $Z(L)\not \subseteq \phi(L)$ there is a maximal subalgebra $K$ of $L$ such that $L=Z(L)+K$ and $L^2\subseteq K$. Hence we have that $Z(L)\cap L^2\subseteq \phi(L)$. Finally, $L=M+L^2$, so choosing $C$ to be a subspace complementary to $L^2\cap M$ in $M$ shows that $L$ splits over $L^2$ and we have case (ii)(b).
\par

Consider now the converse. If (i) holds the converse is clear. So suppose that (ii) holds. Then $L=L^2\dot{+}A$, where $A$ is an abelian subalgebra of $L$. If (a) holds, then $A$ is a maximal subalgebra of $L$ which is abelian. If (b) holds, put $M=A+L^{(2)}+I$. Then $M$ is a maximal subalgebra of $L$ and $M$ is abelian.
\end{proof}

\section{Abelian subalgebras of codimension two}

\begin{lemma}\label{codim2} Let $L=M+Fx$, $M=A+Fe_1$, where $A$ is an abelian ideal of $M$, $M$ is an ideal of $L$ but $[x,A]+[A,x] \not\subseteq A$. Then $Z(M)$ has codimension at most one in $A$, $M^2\subseteq F[e_1,e_2]+Fe_1^2\subseteq A$, where $e_1=[x,e_2]$ or $e_1=[e_2,x]$ and $e_2\in A$.
\end{lemma}
\begin{proof} Put $A=Fe_2+\ldots +Fe_n$, $M=A+Fe_1$. Then we may assume that $e_1=[x,e_2]$ or $e_1=[e_2,x]$. Let
\[ [x,e_j]=\sum_{i=1}^n \alpha_{ji}e_i, \hspace{.2cm} [e_j,x]=\sum_{i=1}^n \beta_{ji}e_i \hbox{ for } j\geq 1,\hspace{.2cm} e_1^2=\sum_{i=1}^n \gamma_i e_i.
\]
{\bf Case 1} Suppose first that $e_1=[x,e_2]$. Then, for $j\geq 2$,
\begin{align} [e_1,e_j]&=[[x,e_2],e_j]=[[x,e_j],e_2]=\alpha_{j1}[e_1,e_2] \\
[e_j,e_1]&=[e_j,[x,e_2]]=[[e_j,x],e_2]=\beta_{j1}[e_1,e_2]
\end{align}
Also $e_1+[e_2,x]\in I$, so
\begin{align}
0=[e_j,e_1]+[e_j,[e_2,x]]=[e_j,e_1]+\beta_{21}[e_j,e_1],
\end{align} whence $[e_j,e_1]=0$ for all $j\geq 2$ or $\beta_{21}=-1$.
\par

Put $v_j=e_j-\alpha_{j1}e_2$ for $j\geq 3$. Then the former, together with (5), shows that $B=Fv_3+\dots +Fv_n\subseteq Z(M)$. So suppose that $\beta_{21}=-1$. Then $[e_1,e_2]=-[e_2,e_1]$ from (6). But now $[x,e_j]+[e_j,x]\in I$, so
\begin{align}
0&=[e_2,[x,e_j]+[e_j,x]]=(\alpha_{j1}+\beta_{j1})[e_2,e_1]
\end{align}
Hence $[e_2,e_1]=0$ or $\alpha_{j1}=-\beta_{j1}$ for $j\geq 2$. Each of these, together with (5) and (6), implies that $B\subseteq Z(M)$ again, so $Z(M)$ has codimension at most one in $A$.
\medskip

\noindent {\bf Case 2} Suppose now that $e_1=[e_2,x]$. Then, for $j\geq 2$,
\begin{align} [e_1,e_j]&=[[e_2,x],e_j]=[e_2,[x,e_j]]=\alpha_{j1}[e_2,e_1] \\
[e_j,e_1]&=[e_j,[e_2,x]]=-[[e_j,x],e_2]=-\beta_{j1}[e_1,e_2]
\end{align}
Putting $j=2$ in (10) gives $[e_2,e_1]=-[e_1,e_2]$, since $\beta_{21}=1$. Also, $e_1+[x,e_2]\in I$, so
\begin{align}
0=[e_j,e_1]+[e_j,[x,e_2]]=[e_j,e_1]+\alpha_{21}[e_j,e_1],
\end{align}
so $[e_j,e_1]=0$ for $j\geq 2$ or $\alpha_{21}=-1$.
\par

Put $v_j=e_j+\alpha_{j1}e_2$ for $j\geq 3$. Then the former, together with (9), gives $B=Fv_3+\ldots +Fv_n\subseteq Z(M)$. So suppose that $\alpha_{21}=-1$. Then (8) still holds, so  $[e_2,e_1]=0$ or $\alpha_{j1}=-\beta_{j1}$ for $j\geq 2$. Each of these, together with (9) and (10), implies that $B\subseteq Z(M)$ again, so $Z(M)$ has codimension at most one in $A$.
\par

Finally, in each case, $M^2\subseteq F[e_1,e_2]+Fe_1^2$, and $e_1^2\in A$ since $A$ is an ideal of $M$.
\end{proof}

\begin{propo}\label{nilpcodim2} Let $L$ be a nilpotent Leibniz algebra of dimension $n$ over a field $F$ of characteristic different from $2$, and let $\alpha(L)=n-2$. Then $\beta(L)=n-2$.
\end{propo}
\begin{proof} Let $A$ be an abelian subalgebra of dimension $n-2$ and assume that it is not an ideal of $L$. Then we can assume that $A$ is an ideal of a subalgebra $M$ of codimension one in $L$, by Theorem \ref{solvcodim1}. Moreover, $M$ is an ideal of $L$, by \cite[Lemma 2.2]{barnes}. It follows that the hypotheses of Lemma \ref{codim2} are satisfied. If $I\not\subseteq Z(M)$ then $I+Z(M)$ is an abelian ideal of dimension $n-2$, by Lemma \ref{codim2}. So assume that $I\subseteq Z(M)$.
\par

Using the same notation as Lemma \ref{codim2}, we have $[x,e_2]=e_1$ and $e_1+[e_2,x]\in Z(M)$, or $[e_2,x]=e_1$ and $e_1+[x,e_2]\in Z(M)$. In the first case,
\[ [e_1,e_2]=-[[e_2,x],e_2]=-[e_2,[x,e_2]]=-[e_2,e_1]
\] and in the second case
\[ [e_2,e_1]=-[e_2,[x,e_2]]=-[[e_2,x],e_2]=-[e_1,e_2],
\] so, in either case, $[e_1,e_2]=-[e_2,e_1]$.
\par

Let $C/Z(M)$ be a minimal ideal of $L/Z(M)$ inside $M/Z(M)$, then $C=F(\lambda e_1+\mu e_2)+Z(M)$. Hence $[\lambda e_1+\mu e_2,x]=\lambda[e_1,x]\mp \mu e_1\in Z(M)$, since $L$ is nilpotent. If $\lambda=0$ then $A$ is an ideal, contradicting our assumption.
\par

If not, then $[e_1,x]\mp \nu e_1\in Z(M)$ where $\nu=\mu/\lambda$. Since $L$ is nilpotent, $\mu=\nu=0$ and $C=Fe_1+Z(M)$. But now,
\begin{align}
 [[e_1,e_2],x]&=[e_1,[e_2,x]]+[[e_1,x],e_2]=\mp e_1^2 \hbox{ and } \nonumber \\
\pm e_1^2=[[e_2,e_1],x]&=[e_2,[e_1,x]]+[[e_2,x],e_1]=\mp e_1^2.
\end{align}  Since $F$ has characteristic different from $2$ we have that $e_1^2=0$ and $C$ is an abelian ideal of dimension $n-2$, whence $\beta(L)=n-2$.

\end{proof}

\begin{lemma}\label{nilrad} Let $L$ be a solvable Leibniz algebra with nilradical $N$. Then $C_L(N)\subseteq N$.
\end{lemma}
\begin{proof} Suppose that $C_L(N) \not\subseteq N$. Then there is a non-trivial abelian ideal $A/(N\cap C_L(N))$ of $L/(N\cap C_L(N))$ inside $C_L(N)/(N\cap C_L(N))$. But now $A^3\subseteq [N,A]=0$, so $A$ is a nilpotent ideal of $L$. It follows that $A\subseteq N\cap C_L(N)$, a contradiction.
\end{proof}

\begin{lemma}\label{cent} Let $J$ be a right ideal of $L$ which is contained in the ideal $K$ of $L$. Then the right centraliser of $J$ in $K$, $C=C_K^r(J)$, is an ideal of $L$. Moreover, if $\dim J=1$ then $C$ has codimension at most one in $K$.
\end{lemma}
\begin{proof} Let $c\in C$, $y\in L$. Then $[j,[c,y]]=[[j,c],y]-[[j,y],c]=0$, since $[j,y]\in J$. Also $[j,[y,c]]=0$, since $[c,y]+[y,c]\in I$, and $[c,y]+[y,c]\in K$.
\par

Now let $J=Fj$. Define $\theta : K \rightarrow J$ by $\theta(k) = [j,k]$. Then Ker$(\theta)= C$ and so $K/C \cong$ Im$(\theta)$ where Im$(\theta) \subseteq J$.
\end{proof}
\medskip

If $U$ is a subalgebra of $L$ the {\em core} of $U$ in $L$, $U_L$, is the largest ideal of $L$ inside $U$.

\begin{theor}\label{solvcodim2} Let $L$ be a solvable Leibniz algebra of dimension $n$ over a field $F$ of characteristic different from $2$ with $\alpha(L)=n-2$. Then one of the following occurs:
\begin{itemize}
\item[(i)] $\beta(L)=n-2$;
\item[(ii)] $L=L_1\dot{+}A$, where $A$ is maximal subalgebra of $L$ that is abelian, $L_1=Fx+Fy$ is a two-dimensional subspace that is not an abelian subalgebra, $[L_1,A]=L_1$, $L^2=L_1+F[x,y]+I$, $L^{(2)}+I=\phi(L)= Z(L)\cap L^2=F[x,y]+I$ and $L^2/(L^{(2)}+I)$ is a chief factor of $L$. In this case $\beta(L)\leq n-3$.
\item[(iii)] $A$ has codimension one in the nilradical, $N$, of $L$, which itself has codimension one in $L$. Moreover, if $L=N+Fx$ and $N=A+Fe_1$ then $N^2\subseteq F[e_1,e_2]+Fe_1^2 \subseteq A$ where $e_1=[x,e_2]$ or $e_1=[e_2,x]$ and $e_2\in A$, $Z(N)$ is an abelian ideal of maximal dimension and $\beta(L)=n-3$. Moreover, $I\subseteq Z(N)$.
\end{itemize}
\end{theor}
\begin{proof} Suppose that (i) does not hold and that $A$ is an abelian subalgebra of $L$ of dimension $n-2$.
\par

\noindent {\bf Case 1}: Suppose that $A$ is a maximal subalgebra of $L$. Then $A$ is a Cartan subalgebra of $L$ and (ii) (a) or (ii) (b) of Theorem \ref{solv} hold.
\par

{\bf Case 1 (a)}: Suppose that Theorem \ref{solv} (ii) (a) holds. Let $L=A\dot{+} L_1$ be the Fitting decomposition of $L$ relative to $A$. Then $L_1\subseteq L^2=I$ and $\dim L_1=2$. But $L_1$ is abelian, and $[L_1,A]\subseteq L_1$, $[A,L_1]\subseteq [L,I]=0$, so $L_1$ is an ideal of $L$. Now $L^2\subseteq L_1\subseteq L^2$ so $L^2$ is a two-dimensional minimal ideal of $L$ over which $L$ splits.
\par

Let $\theta : L \rightarrow$ Der($L^2$) be defined by $\theta(x)=R_x|_{L^2}$ for all $x\in L$. Then $\theta$ is a homomorphism with kernel $C_L^r(L^2)$, the right centraliser in $L$ of $L^2$, which is an ideal of $L$, by Lemma \ref{cent}. It follows that $L/C_L^r(L^2)\cong D$, where $D$ is an abelian subalgebra of Der$L^2\cong gl(2,F)$. It follows from Schur's theorem on commuting matrices (see \cite{jac}) that $\dim L/C_L^r(L^2)\leq [2^2/4]+1=2$. But then $C_L^r(L^2)=L^2\oplus (A\cap C_L^r(L^2))$ is an abelian ideal of dimension $n-2$, since $[A\cap C_L^r(L^2),L^2]\subseteq [L,I]=0$. We therefore have (i), a contradiction.
\par

{\bf Case 1 (b)}:  Suppose that Theorem \ref{solv} (ii) (b) holds. Let $L=A\dot{+} L_1$ be the Fitting decomposition of $L$ relative to $A$. Then $L^{(2)}+I=\phi(L)\subseteq A$ and $(L^{(2)}+I)\dot{+}L_1\subseteq L^2$. If $S$ is a subalgebra of $L$, denote by $\overline{S}$ its image under the canonical homomorphism onto $L/(L^{(2)}+I)$. Then, as in Theorem \ref{solv}, $\overline{L_1}=\overline{L^2}$, so $L^2=L_1\dot{+}(L^{(2)}+I)$. Put $L_1=Fx+Fy$. Then
\[ L^2=Fx+Fy+F[x,y]+F[y,x]+Fx^2+Fy^2+I=Fx+Fy+F[x,y]+I,
\] and $L^{(2)}+I=F[x,y]+I$. If $L_1$ is an abelian subalgebra then it is a right ideal and $C_L^r(L_1)$ is an abelian ideal of dimension $n-2$, a contradiction. We thus have case (ii).
\par

Let $C$ be a maximal abelian ideal of $L$. Then $L^{(2)}+I\subseteq Z(L)\subseteq C$, but $L^2\not\subseteq C$ and $L^2\not\subseteq C$ since $L_1$ is not abelian.Thus $L^2\cap C= L^{(2)}+I$. If $\dim C=n-2$, then
\begin{align*}
\dim (L^2+C)&=\dim L^2+ \dim C - \dim C\cap L^2 \\
&=2+\dim (L^{(2)}+I) + n-2- \dim (L^{(2)}+I)=n.
\end{align*} It follows that $L=L^2+C$. But now $L^2\subseteq L^{(2)}+(L^2\cap C)=L^{(2)}+I\subseteq C$, a contradiction. Hence $\beta(L)\leq n-3$.

{\bf Case 2}: So suppose now that $A$ is not a maximal subalgebra of $L$. Then $A\subset M\subset L$ where $M$ is a subalgebra with $\dim M=n-1$. We may assume that $A$ is an ideal of $M$, by Theorem \ref{solvcodim1}.
\par

{\bf Case 2 (a)}: Suppose that $A$ does not act nilpotently on $L$. Then $L=L_1\dot{+} M$ is the Fitting decomposition of $L$ relative to $A$. Put $L_1=Fx$. Then $[L_1,A]=L_1$ so there exists $a\in A$ such that $[x,a]=x$, and $[x,M]\subseteq Fx$.
\par

(i) Suppose first that  $A+Fx^2=M$. Now $[x^2,a]=[x,[x,a]]+[[x,a],x]=2x^2$. But $[x^2,a]\in A$, which is a contradiction.
\par

(ii) Suppose now that $L=M+Fx^2$, so $x=m+\lambda x^2$ for some $m\in M$, $0\neq \lambda\in F$. Then $[x^2, M]\subseteq [x,[x,M]]+[[x,M],x]\subseteq Fx^2$, so $Fx^2$ is an ideal of $L$. Put $B=A+Fx^2$. Then $B$ is an ideal of $L$. It follows from Lemma \ref{cent} that $C=C_B^r(Fx^2)$ is an ideal of codimension at most one in $B$. Furthermore, if $c_1=a_1+\lambda x^2, c_2=a_2+\mu x^2\in C$, then $[c_1,c_2]=\mu[a_1,x^2]=0$, so $C$ is abelian and $\beta(L)=n-2$, contradiction.
\par

(iii) Finally, suppose that $x^2\in A$. Then $0=[x^2,a]=[x,[x,a]]+[[x,a],x]=2x^2$, so $x^2=0$, so $Fx$ is a right ideal of $L$. Put $B=Fx+A$. Then $[A,x]=[A,[x,a]]\subseteq [[A,x],a]\subseteq [M+Fx,a]\subseteq A+Fx$ which implies that $[M,x]=[M,[x,a]]\subseteq [[M,x],a]+[[M,a],x]\subseteq [M+Fx,a]+[A,x]\subseteq A+Fx$. It follows that $B$ is an ideal of $L$ and that $C=C_B^r(Fx)$ is an ideal of $L$ of codimension at most one in $B$. Let $c_1=a_1+\lambda x, c_2=a_2+\mu x\in C$. Then $[c_1,c_2]=\mu[a_1,x]$. Let $[a_1,x]=a_3+\nu x$. If $\nu=0$, then $[a_1,x]=[a_1,[x,a]]=[[a_1,x],a]=0$. If $\nu\not = 0$, then $\nu[a_1,x]=[a_1,[a_1,x]]=-[[a_1,x],a_1]=-\nu[x,a_1]=-\nu[x,c_1]=0$. Hence $C$ is abelian and $\beta(L)=n-2$, a contradiction.
\par

{\bf Case 2 (b)}: So suppose now that $A$ does act nilpotently on $L$. Then there is $k\geq 0$ such that $R_A^k(L)\not \subseteq M$ but $R_A^{k+1}(L)\subseteq M$. Let $x\in R_A^k(L)\setminus M$, so $L=M+Fx$ and $[x,A]\subseteq M$.
\par

Suppose that $M$ is not an ideal of $L$. Then $M_L$ has codimension $1$ in $M$, by \cite[Theorem 3.2]{qa}. Now $A\neq M_L$, since (i) does not hold, so $M=A+M_L$. We have $[A,x]\subseteq M+I$. If $I\subseteq M$ this gives $[A,x]\subseteq M$; if $I\not\subseteq M$ then $L=M+I$ and $[A,x]\subseteq [A, M+I] \subseteq M$ again. It follows that $M$ is an ideal of $L$.
\par

Let $N$ be the nilradical of $L$. If $N\subseteq A$ then $A\subseteq C_L(N)\subseteq N$, by Lemma \ref{nilrad}, so $N=A$, contradicting the fact that (i) does not hold. If $A\subseteq N$, then $N=L$ or we can assume that $N=M$. If $A\not\subseteq N$ and $N\not\subseteq A$ , then either $L=A+N$, in which case $L$ is nilpotent, or we can assume that $A+N=M$, in which case $M$ is a nilpotent ideal of $L$, and so $M=N$.
\par

If $L$ is nilpotent then $\beta(L)=n-2$, by Proposition \ref{nilpcodim2}. If not, then $N=A+Fe_1$ and $N^2\subseteq	F[e_1,e_2]+Fe_1^2 \subseteq A$ where $e_1=[x,e_2]$ or $e_1=[e_2,x]$, $e_2\in A$, and $\dim Z(N)=n-3$., by Lemma \ref{codim2}.  If $I\not \subseteq Z(N)$ then $I+Z(N)$ must be an abelian ideal of dimension $n-2$, a contradiction, so $I\subseteq Z(N)$.
\end{proof}

\begin{coro}\label{supsolvcodim2} Let $L$ be a supersolvable Leibniz algebra of dimension $n$ over a field $F$ of characteristic different from $2$ and let $\alpha(L)=n-2$. Then $\beta(L)=n-2$.
\end{coro}
\begin{proof} Let $L$ be supersolvable and suppose that case (i) of Theorem \ref{solvcodim2} does not occur. Suppose first that case (ii) occurs. Then $L=L^2+A$ and $L^2=Z(L)\cap L^2+Fx$ for some $x\in L^2$. It follows that $L=A+Fx$, since $Z(L)\subseteq A$, a contradiction.
\par

So, suppose that case (iii) occurs. Then $\dim Z(N)=n-3$. Since $L$ is supersolvable, there is an ideal $C\subset  N$ of $L$ with $\dim (C/Z(N)) = 1$. Then we can argue similarly to Proposition \ref{nilpcodim2}. Putting $C=F(\lambda e_1+\mu e_2)+Z(N)$ we have that $\lambda[e_1,x]\mp \mu e_1=\gamma(\lambda e_1+\mu e_2)+z$ for some $\gamma \in F$, $z\in Z(N)$. If $\lambda=0$ then $A$ is an ideal and $\beta(L)=n-2$,
\par

If not, then $[e_1,x]=\delta e_1+\epsilon e_2+z$ where $\delta, \epsilon \in F$, $z\in Z(N)$. As in Proposition \ref{nilpcodim2}, $[e_1,e_2]=-[e_2,e_1]$ which  yields that
\begin{align*}
 [[e_1,e_2],x]&=[e_1,[e_2,x]]+[[e_1,x],e_2]=\mp e_1^2+\delta [e_1,e_2] \hbox{ and } \\
\pm e_1^2-\delta [e_1,e_2]&=[[e_2,e_1],x]=[e_2,[e_1,x]]+[[e_2,x],e_1]=\delta [e_2,e_1]\mp e_1^2.
\end{align*}  Since $F$ has characteristic different from $2$ we have that $e_1^2=0$ and $C$ is an abelian ideal of dimension $n-2$, whence $\beta(L)=n-2$.

\end{proof}
\medskip

Note that all of the non-Lie Leibniz algebras described in Theorem \ref{solvcodim2} are in cases (i) and (iii), as the following result shows.

\begin{propo} The algebras of type (ii) in Theorem \ref{solvcodim2} are all Lie algebras. (Here we are retaining the assumption that the underlying field has characteristic different from $2$.)
\end{propo}
\begin{proof} If $[x,A]\subseteq Fx$ then $Fx+L^{(2)}+I/(L^{(2)}+I)$ is a proper non-trivial ideal of $L^2/(L^{(2)}+I)$, contradicting the fact that the latter is a chief factor of $L$. Hence there exists an $a\in A$ such that $x$ and $[x,a]$ are linearly independent. Put $y=[x,a]$, so $L_1=Fx+Fy$. Now $[L_1,A]=L_1$, so there exist $u\in L_1, a_1\in A$ such that $x=[u,a_1]$. Then, for any $a_2\in A$,
\[
[a_2,x]=[a_2,[u,a_1]]=[[a_2,u],a_1] \hbox{ and } [x,a_2]=[[u,a_1],a_2]=[[u,a_2],a_1].
\] But $[[a_2,u]+[u,a_2],a_1]=0$, so $[a_2,x]=-[x,a_2]$. Also,
\[ [a_2,y]=[a_2,[x,a]]=[[a_2,x],a]=-[[x,a_2],a]=-[[x,a],a_2]=-[y,a_2],
\] so elements of $L_1$ anti-commute with elements of $A$.
\par

Next we have $[x,y]=[x,[x,a]]=-[[x,a],x]=-[y,x]$, so it remains to examine $x^2$ and $y^2$. Using the fact that $L_1^2\subseteq L^{(2)}\subseteq Z(L)$ we have
\begin{align}
y^2&=[y,[x,a]]=-[[y,a],x] \hbox{ and } \\
y^2&=-[[a,x],y]=-[[a,y],x]=[[y,a],x]
\end{align} Adding (1) and (2) gives $y^2=0.$
Similarly,
\begin{align}
x^2&=[x,[u,a_1]]=-[[x,a_1],u] \hbox{ and } \\
x^2&=-[[a_1,u],x]=-[[a_1,x],u]=[[x,a_1],u]
\end{align} Adding (3) and (4) gives $x^2=0$
\end{proof}
\medskip

Here, we show examples of Theorem 3.5 for cases $(ii)$ and $(iii)$.

\begin{ex} Let $L=Fe_1+F e_2+F e_3+F e_4$ and non-zero products
\[
[e_1,e_2] = -[e_2,e_1] = e_3, [e_1,e_3] =-[e_3,e_1]= - e_2, [e_2,e_3]=-[e_3,e_2] =- e_4.
\]
 This is a Lie algebra over the real number field $F$ and is as described in Theorem \ref{solvcodim2} (ii).  For $L^2 = Fe_2 + Fe_3 + Fe_4$, $I=L^{(2)}=Z(L)=\phi(L)=Fe_4$, and $L^2/(L^{(2)}+I)$ is a two-dimensional chief factor of $L$. This algebra has $\alpha(L) = 2$ and $\beta(L) = 1$. We could take $A = Fe_1 + Fe_4$, $L_1=Fe_2+Fe_3$ and $Fe_4$ is the unique maximal abelian ideal of $L$.
\end{ex}

\begin{ex}
Let $L=Fe_1+Fe_2+Fe_3+Fe_4+Fe_5$ where $[e_1,e_2]=-[e_2,e_1]=e_3,$  $[e_1,e_3]=-[e_3,e_1]=-e_2$,  $[e_2,e_3]=-[e_3,e_2]=e_4$, $[e_1,e_5]=e_4$
and all other products zero. This is a non-Lie Leibniz algebra over the real number field $F$ and is as described in Theorem \ref{solvcodim2} (iii).
Then $A=Fe_3+Fe_4+Fe_5$ is an abelian subalgebra of maximal dimension, $N=Fe_2+Fe_3+Fe_4+Fe_5$ is the nilradical of $L$, $N^2=Fe_4 \subset A$ and $Z(N)=Fe_4+Fe_5$ is the unique abelian ideal of maximal dimension.
\end{ex}

\section{Abelian subalgebras in filiform Leibniz algebras}

%
%

\begin{lemma}\label{lemmafiliform}

Let $L$ be an $n$-dimensional $p$-filiform Leibniz algebra. Then there exists a basis $e_1,\ldots,e_n$ of $L$ such that $[e_i,e_1]=e_{i+1}$, $\forall p+1 \leq i \leq n-1$.

\end{lemma}

\begin{proof}
In case of null-filiform Leibniz algebras, there exists a basis $e_1,\ldots,e_n$ of $L$ such that
\[ [e_i,e_1]=e_{i+1}, \quad [e_i,e_j]=0, \quad 1 \leq i \leq n-1, \, j \geq 2
\]
This was proved in \cite[Theorem 3.2]{book}. Moreover, this type of algebras  are mutually isomorphic in each dimension (see \cite{ao}).
Let $L$ be a $p$-filiform Leibniz algebra with $p \geq 1$. Then we can choose a basis $e_1,\ldots,e_n$ of $L$ such that
$$e_1, \ldots, e_{p+1} \in L \setminus L^2 \quad {\rm and} \quad e_{p+i} \in L^{i} \setminus L^{i+1}, \,\, {\rm for} \,\,2 \leq i \leq n-p$$
Since $e_{p+2} \in L^2 \setminus L^3$, $\exists e_i, e_j \in L$ such that $[e_i,e_j]=a_{ij} e_{p+2}$, where $1 \leq i,j \leq n$ and $a_{ij} \neq 0$. We can suppose, without loss of generality, that $i=p+1, j=1$ and $a_{ij}=1$. In this way, we get $[e_{p+1},e_1]=e_{p+2}$.
We conclude the proof by following this procedure for every vector $e_{p+i}$, for $i=3, \ldots, n-p-1$ and redefining $e_1$ as a suitable linear combination of $e_1,\ldots,e_{p+1}$.

\end{proof}

\begin{propo}
Let $L$ be an $n$-dimensional $k$-abelian $p$-filiform Leibniz algebra. Then $L^k$ is the unique abelian ideal of maximal dimension and, therefore, $\alpha(L)=\beta(L)=n-k+1$

\end{propo}

\begin{proof}
According to Lemma \ref{lemmafiliform}, there exists a basis $e_1,\ldots,e_n$ of $L$ such that
\[ [e_i,e_1]=e_{i+1}, \quad p+1 \leq i \leq n-1
\]
In this case, $L^k=span\{e_{p+k},\ldots,e_n\}$ for $2 \leq k \leq n-p$. Let us suppose that $A$ is an abelian ideal of $L$ which is not contained in $L^k$. There is an $e=\sum_{i=1}^n \alpha_i e_i \in A$ but $e \notin L^k$. Consequently, there exists $ j$ with $1 \leq j < p+k$ and $\alpha_j \neq 0$. We choose the minimal $j$ satisfying that condition. Then,
$$R_{e_1} (e)=[e,e_1]=\alpha_j e_{j+1} + \sum_{i=j+1}^{n-1} \alpha_i e_{i+1} \in A$$
In fact,
$$R^{\ell}_{e_1} (e)=\alpha_j e_{j+\ell} + \sum_{i=j+1}^{n-\ell} \alpha_i e_{i+\ell} \in A, \quad \forall \, 0 \leq \ell \leq n-j$$
It follows that $\alpha_j e_{j+\ell} \in A$, for $0 \leq \ell \leq n-j$. Since $j < k$ and $\alpha_j e_j \in A$, we conclude that $L^{k-1} \subset A$ and, therefore, $A$ is not abelian, which is a contradiction.
\end{proof}

\begin{ex}
Filiform Leibniz algebras were classified in \cite[Theorem 2.5]{GO} by families $\mu_1^{\overline{\alpha},\theta}, \mu_2^{\overline{\beta},\gamma}$ and
$\mu_3^{\alpha,\beta,\gamma}$. In case of $L=\mu_1^{\overline{\alpha},\theta}$ or $L=\mu_2^{\overline{\beta},\gamma}$, we have that $L^2$ is the unique abelian ideal of maximal dimension. If $L=\mu_3^{\alpha,\beta,\gamma}$, then $L^{n-3}=Fe_{n-2}+Fe_{n-1}+Fe_n$ is the unique abelian ideal of maximal dimension.
\end{ex}

\section{Tables of $\alpha$ and $\beta$ for Leibniz algebras}

The values of $\alpha$ and $\beta$ invariants were already obtained for Leibniz algebras up to dimension $4$ in \cite[Proposition 4.1]{C}. and \cite[Propositions 1-3]{CNT} by using the classifications given in \cite{A,Canete,CILL}. In Tables \ref{alphaybetaleibnizdim5_nilrad4_I}-\ref{alphaybetaleibnizdim5_nilrad3_II}, we have computed the value of those invariants for $5$-dimensional solvable Leibniz algebras using the classifications given in \cite{KS,KRH}. Finally, Tables \ref{alphaybetaleibnizdim6nilpotentI}-\ref{alphaybetaleibnizdim6nilpotentV} contain the calculations for $6$-dimensional nilpotent Leibniz algebras with a $4$-dimensional derived algebra, whose classification can be found in \cite{Demir}.

\begin{table}[htp] \caption{$5$-dimensional solvable Leibniz algebra with $4$-dimensional nilradical (I).}
\label{alphaybetaleibnizdim5_nilrad4_I}
\begin{center}
\begin{tabular}{|c|c|c|c|}
\hline
$L$ & Products & $\alpha(L)$ & $\beta(L)$ \\
\hline
$L_{1}^{\alpha}$ & \begin{tabular}{c}
$[e_1,e_1]=e_4, [e_1,e_5]=e_1-\alpha e_2, [e_5,e_1]=-e_1+\alpha e_2$ \\
$[e_1,e_2]=\alpha e_4, [e_2,e_5]=\alpha e_1+e_2, [e_5,e_2]=-\alpha e_1-e_2$\\
$[e_2,e_1]=-\alpha e_4, [e_3,e_5]=e_3, [e_5,e_3]=-e_3$ \\
$[e_2,e_2]=e_4, [e_4,e_5]=2e_4, [e_3,e_3]=e_4$
\end{tabular}  & $2$  & $2$  \\
\hline
$L_{2}$ & \begin{tabular}{c}
$[e_1,e_1]=e_4, [e_1,e_5]=e_1-i e_2, [e_5,e_1]=-e_1+i e_2$ \\
$[e_1,e_2]=i e_4, [e_2,e_5]=i e_1+e_2, [e_5,e_2]=-i e_1-e_2+e_4$\\
$[e_2,e_1]=-i e_4, [e_3,e_5]=e_3, [e_5,e_3]=-e_3$ \\
$[e_2,e_2]=e_4, [e_4,e_5]=2e_4, [e_3,e_3]=e_4$
\end{tabular}  & $2$  & $2$  \\
\hline
$L_{3}$ & \begin{tabular}{c}
$[e_1,e_2]=e_4, [e_1,e_5]=e_1- e_2, [e_5,e_1]=-e_1+e_2$ \\
$[e_1,e_3]=e_4, [e_2,e_5]=e_2+e_3, [e_5,e_2]=-e_2-e_3$\\
$[e_2,e_1]=- e_4, [e_3,e_5]=e_3, [e_5,e_3]=-e_3$ \\
$[e_2,e_2]=e_4, [e_4,e_5]=2e_4, [e_3,e_3]=e_4$
\end{tabular}  & $2$  & $1$  \\
\hline
$L_{4}^{\alpha}$ & \begin{tabular}{c}
$[e_1,e_5]=e_1, [e_5,e_1]=-e_1, [e_1,e_2]=e_3$ \\
$[e_2,e_5]=\alpha e_2,[e_5,e_2]=-\alpha e_2, [e_2,e_1]=e_4$\\
$[e_3,e_5]=(1+\alpha)e_3, [e_5,e_3]=-e_3+\alpha e_4,$ \\
$ [e_4,e_5]=(1+\alpha)e_4, [e_5,e_4]=e_3-\alpha e_4$
\end{tabular}  & $3$  & $3$  \\
\hline
$L_{5}$ & \begin{tabular}{c}
$[e_1,e_2]=e_3, [e_1,e_5]=e_1+e_3, [e_5,e_1]=-e_1+e_4$ \\
$[e_2,e_1]= e_4,[e_3,e_5]=e_3, [e_5,e_3]=-e_3,$ \\
$ [e_4,e_5]=e_4, [e_5,e_4]=e_3$
\end{tabular}  & $3$  & $3$  \\
\hline
$L_{6}$ & \begin{tabular}{c}
$[e_1,e_1]=e_4, [e_1,e_5]=e_1+e_2, [e_2,e_5]=e_2$ \\
$[e_1,e_2]= e_3,[e_5,e_1]=-e_1-e_2, [e_5,e_2]=-e_2,$ \\
$[e_2,e_1]=-e_3, [e_3,e_5]=2e_3,$ \\ $[e_5,e_3]=-2e_3, [e_4,e_5]=2e_4$
\end{tabular}  & $3$  & $3$  \\
\hline
$L_{7}^{\alpha}$ & \begin{tabular}{c}
$[e_1,e_1]=e_4, [e_1,e_5]=e_1, [e_5,e_1]=-e_1, [e_4,e_5]=2e_4$ \\
$[e_1,e_2]= e_3,[e_2,e_5]=\alpha e_2, [e_3,e_5]=(1+\alpha)e_3,$ \\
$[e_2,e_1]=-e_3, [e_5,e_2]=-\alpha e_2, [e_5,e_3]=-(1+\alpha)e_3$
\end{tabular}  & $3$  & $3$  \\
\hline
$L_{8}$ & \begin{tabular}{c}
$[e_1,e_1]=e_4, [e_2,e_5]=e_2, [e_5,e_1]=-e_1$, \\
$[e_1,e_2]= e_3,[e_3,e_5]=e_3, [e_5,e_3]=-e_3$, \\
$[e_2,e_1]=-e_3$
\end{tabular}  & $3$  & $3$  \\
\hline
$L_{9}^{\alpha}$ & \begin{tabular}{c}
$[e_1,e_1]=e_4, [e_2,e_5]=e_2, [e_5,e_1]=\alpha e_4, $ \\
$[e_1,e_2]=e_3, [e_3,e_5]= e_3, [e_5,e_2]=- e_2,$ \\
$[e_2,e_1]=-e_3, [e_5,e_3]=-e_3$
\end{tabular}  & $3$  & $3$  \\
\hline
\end{tabular}
\end{center}
\end{table}

\begin{table}[htp] \caption{$5$-dimensional solvable Leibniz algebra with $4$-dimensional nilradical (II).}
\label{alphaybetaleibnizdim5_nilrad4_II}
\begin{center}
\begin{tabular}{|c|c|c|c|}
\hline
$L$ & Products & $\alpha(L)$ & $\beta(L)$ \\
\hline
$L_{10}^{\alpha}$ & \begin{tabular}{c}
$[e_1,e_1]=e_4, [e_2,e_5]=e_2, [e_5,e_1]=\alpha e_4, $ \\
$[e_1,e_2]=e_3, [e_3,e_5]= e_3, [e_5,e_2]=-e_2,$ \\
$[e_2,e_1]=-e_3, [e_5,e_5]=e_4  [e_5,e_3]=-e_3$
\end{tabular}  & $3$  & $3$  \\
\hline
$L_{11}^{\alpha,\delta}$ & \begin{tabular}{c}
$[e_1,e_1]=e_4, [e_2,e_5]=e_2+e_3, [e_5,e_1]=\alpha e_4, $ \\
$[e_1,e_2]=e_3, [e_3,e_5]= e_3, [e_5,e_2]=-e_2-e_3,$ \\
$[e_2,e_1]=-e_3, [e_5,e_5]=\delta e_4  [e_5,e_3]=-e_3$
\end{tabular}  & $3$  & $3$  \\
\hline
$L_{12}$ & \begin{tabular}{c}
$[e_1,e_1]=e_3, [e_2,e_5]=e_2, [e_5,e_2]=-e_2, $ \\
$[e_1,e_2]=e_4,  [e_4,e_5]=e_4$
\end{tabular}  & $3$  & $3$  \\
\hline
$L_{13}$ & \begin{tabular}{c}
$[e_1,e_1]=e_3, [e_2,e_5]=e_2, [e_5,e_2]=-e_2, $ \\
$[e_1,e_2]=e_4,  [e_4,e_5]=e_4, [e_5,e_5]=e_3$
\end{tabular}  & $3$  & $3$  \\
\hline
$L_{14}^{\alpha}$ & \begin{tabular}{c}
$[e_1,e_1]=e_3, [e_2,e_5]=e_2, [e_5,e_1]=-e_3, $ \\
$[e_1,e_2]=e_4,  [e_4,e_5]=e_4, [e_5,e_2]=-e_2, [e_5,e_5]=\alpha e_3$
\end{tabular}  & $3$  & $3$  \\
\hline
$L_{15}$ & \begin{tabular}{c}
$[e_1,e_2]=e_4, [e_1,e_5]=e_1, [e_5,e_1]=-e_1, $ \\
$[e_2,e_1]=e_4,  [e_2,e_5]=e_2, [e_5,e_2]=-e_2,$ \\
$[e_2,e_2]= e_4, [e_3,e_5]=2e_3, [e_4,e_5]=2e_4$
\end{tabular}  & $3$  & $3$  \\
\hline
$L_{16}$ & \begin{tabular}{c}
$[e_1,e_2]=e_4, [e_1,e_5]=e_1, [e_5,e_1]=-e_1, $ \\
$[e_2,e_1]=-e_4,  [e_2,e_5]=-e_2, [e_5,e_2]=e_2,$ \\
$[e_3,e_3]= e_4, [e_3,e_5]=e_4$
\end{tabular}  & $2$  & $2$  \\
\hline
$L_{17}$ & \begin{tabular}{c}
$[e_1,e_2]=e_4, [e_1,e_5]=e_1, [e_5,e_1]=-e_1, $ \\
$[e_2,e_1]=-e_4,  [e_2,e_5]=-e_2, [e_5,e_2]=e_2,$ \\
$[e_3,e_3]= e_4, [e_3,e_5]=e_4, [e_5,e_3]=e_4$
\end{tabular}  & $2$  & $2$  \\
\hline
$L_{18}^{\alpha}$ & \begin{tabular}{c}
$[e_1,e_2]=e_4, [e_1,e_5]=e_1, [e_5,e_1]=-e_1, $ \\
$[e_2,e_1]=-e_4,  [e_2,e_5]=-e_2, [e_5,e_2]=e_2,$ \\
$[e_3,e_3]= e_4, [e_3,e_5]=e_4, [e_5,e_3]=\alpha e_4,[e_5,e_5]=e_4$
\end{tabular}  & $2$  & $2$  \\
\hline
\end{tabular}
\end{center}
\end{table}

\begin{table}[htp] \caption{$5$-dimensional solvable Leibniz algebra with $3$-dimensional nilradical (I).}
\label{alphaybetaleibnizdim5_nilrad3_I}
\begin{center}
\begin{tabular}{|c|c|c|c|}
\hline
$L$ & Products & $\alpha(L)$ & $\beta(L)$ \\
\hline
$L_{1}$ & \begin{tabular}{c}
$[e_1,e_2]=e_3, [e_2,e_1]=-e_3, [e_1,e_4]=e_1, $ \\
$[e_3,e_4]=e_3,  [e_4,e_1]=-e_1, [e_4,e_3]=-e_3,$ \\
$[e_5,e_2]= e_2, [e_3,e_5]=e_3,$ \\
$ [e_5,e_2]=-e_2,[e_5,e_3]=-e_3$
\end{tabular}  & $2$  & $2$  \\
\hline
$L_{2}$ & \begin{tabular}{c}
$[e_2,e_1]=e_3, [e_1,e_4]=e_1, [e_2,e_5]=e_2$ \\
$[e_4,e_1]=-e_1, [e_3,e_4]=e_3,[e_3,e_5]=e_3$
\end{tabular}  & $2$  & $2$  \\
\hline
$L_{3}$ & \begin{tabular}{c}
$[e_1,e_1]=e_3, [e_1,e_4]=e_1, [e_4,e_1]=-e_1$ \\
$[e_3,e_4]=2e_3, [e_2,e_5]=e_2$
\end{tabular}  & $2$  & $2$  \\
\hline
$L_{4}$ & \begin{tabular}{c}
$[e_1,e_1]=e_3, [e_1,e_4]=e_1, [e_4,e_1]=-e_1$ \\
$[e_3,e_4]=2e_3, [e_2,e_5]=e_2,[e_5,e_2]=-e_2$
\end{tabular}  & $2$  & $2$  \\
\hline
\begin{tabular}{c}$L_{5}^{\mu_1,\mu_2}$ \\ $(\mu_1\neq0)$ \end{tabular} & \begin{tabular}{c}
$[e_1,e_4]=e_1, [e_3,e_4]=\mu_1 e_3, [e_2,e_5]=e_2$ \\
$[e_3,e_5]=\mu_2 e_3, [e_4,e_1]=-e_1,[e_4,e_3]=-\mu_1 e_3$ \\
$[e_5,e_2]=-e_2, [e_5,e_3]=-\mu_2 e_3$
\end{tabular}  & $3$  & $3$  \\
\hline
$L_{6}^{\mu_1,\mu_2}$ & \begin{tabular}{c}
$[e_1,e_4]=e_1, [e_3,e_4]=\mu_1 e_3, [e_2,e_5]=e_2$ \\
$[e_3,e_5]=\mu_2 e_3, [e_4,e_1]=-e_1,[e_5,e_2]=-e_2$
\end{tabular}  & $3$  & $3$  \\
\hline
\begin{tabular}{c}$L_{7}^{\mu}$ \\ $(\mu \neq0)$ \end{tabular} & \begin{tabular}{c}
$[e_1,e_4]=e_1, [e_2,e_5]=e_2, [e_3,e_5]=\mu e_3,$ \\
$[e_5,e_2]=-e_2, [e_5,e_3]=-\mu e_3$
\end{tabular}  & $3$  & $3$  \\
\hline
$L_{8}^{\mu_1,\mu_2}$ & \begin{tabular}{c}
$[e_1,e_4]=e_1, [e_3,e_4]=\mu_1 e_3, [e_2,e_5]=e_2$ \\
$[e_3,e_5]=\mu_2 e_3, [e_5,e_2]=-e_2$
\end{tabular}  & $3$  & $3$  \\
\hline
$L_{9}^{\mu_1,\mu_2}$ & \begin{tabular}{c}
$[e_1,e_4]=e_1, [e_3,e_4]=\mu_1 e_3, [e_2,e_5]=e_2$ \\
$[e_3,e_5]=\mu_2 e_3$
\end{tabular}  & $3$  & $3$  \\
\hline
\begin{tabular}{c}$L_{10}^{\lambda_1,\lambda_2,\lambda_3,\lambda_4}$ \\
$\exists i \,\, | \,\, \lambda_i \neq 0$ \end{tabular} & \begin{tabular}{c}
$[e_1,e_4]=e_1, [e_2,e_5]=e_2,[e_4,e_1]=-e_1,$ \\
$[e_5,e_2]=-e_2, [e_4,e_4]=\lambda_1 e_3, [e_5,e_4]=\lambda_2 e_3$ \\
$[e_4,e_5]=\lambda_3e_3, [e_5,e_5]=\lambda_4 e_3$
\end{tabular}  & $3$  & $3$  \\
\hline
\begin{tabular}{c}$L_{11}^{\lambda_1,\lambda_2,\lambda_3,\lambda_4}$ \\
$\exists i \,\, | \,\, \lambda_i \neq 0$ \end{tabular} & \begin{tabular}{c}
$[e_1,e_4]=e_1, [e_2,e_5]=e_2,[e_5,e_2]=-e_2,$ \\
$[e_4,e_4]=\lambda_1 e_3, [e_5,e_4]=\lambda_2 e_3$ \\
$[e_4,e_5]=\lambda_3e_3, [e_5,e_5]=\lambda_4 e_3$
\end{tabular}  & $3$  & $3$  \\
\hline
\begin{tabular}{c}$L_{12}^{\lambda_1,\lambda_2,\lambda_3,\lambda_4}$ \\
$\exists i \,\, | \,\, \lambda_i \neq 0$ \end{tabular} & \begin{tabular}{c}
$[e_1,e_4]=e_1, [e_2,e_5]=e_2,[e_4,e_4]=\lambda_1 e_3,$ \\ $[e_5,e_4]=\lambda_2 e_3, [e_4,e_5]=\lambda_3e_3, [e_5,e_5]=\lambda_4 e_3$
\end{tabular}  & $3$  & $3$  \\
\hline
$L_{13}$ & \begin{tabular}{c}
$[e_1,e_4]=e_1, [e_2,e_5]=e_2,[e_3,e_4]= e_3,$ \\
$[e_4,e_1]=-e_1, [e_5,e_1]=-e_3$
\end{tabular}  & $3$  & $3$  \\
\hline
$L_{14}$ & \begin{tabular}{c}
$[e_1,e_4]=e_1, [e_3,e_4]=e_3,[e_4,e_1]= -e_1,$ \\
$[e_5,e_1]=e_3, [e_2,e_5]=e_2, [e_5,e_2]=-e_2$
\end{tabular}  & $3$  & $3$  \\
\hline
\end{tabular}
\end{center}
\end{table}

\begin{table}[htp] \caption{$5$-dimensional solvable Leibniz algebra with $3$-dimensional nilradical (II).}
\label{alphaybetaleibnizdim5_nilrad3_II}
\begin{center}
\begin{tabular}{|c|c|c|c|}
\hline
$L$ & Products & $\alpha(L)$ & $\beta(L)$ \\
\hline
$L_{15}$ & \begin{tabular}{c}
$[e_1,e_4]=e_1, [e_2,e_4]=e_2,$ \\
$[e_1,e_5]= e_2, [e_3,e_5]=e_3$
\end{tabular}  & $3$  & $3$  \\
\hline
$L_{16}$ & \begin{tabular}{c}
$[e_1,e_4]=e_1, [e_2,e_4]=e_2,$ \\
$[e_1,e_5]= e_2, [e_3,e_5]=e_3,[e_5,e_3]=-e_3$
\end{tabular}  & $3$  & $3$  \\
\hline
$L_{17}$ & \begin{tabular}{c}
$[e_1,e_4]=e_1, [e_2,e_4]=e_3,$ \\
$[e_1,e_5]= e_2, [e_3,e_5]=e_3,$ \\
$[e_4,e_1]=-e_1, [e_4,e_2]=-e_2, [e_5,e_1]=-e_2$
\end{tabular}  & $3$  & $3$  \\
\hline
$L_{18}$ & \begin{tabular}{c}
$[e_1,e_4]=e_1, [e_2,e_4]=e_2,$ \\
$[e_1,e_5]= e_2, [e_3,e_5]=e_3,[e_4,e_1]=-e_1,$ \\
$ [e_4,e_2]=-e_2, [e_5,e_1]=-e_2, [e_5,e_3]=-e_3$
\end{tabular}  & $3$  & $3$  \\
\hline
$L_{19}^{\mu}$ & \begin{tabular}{c}
$[e_1,e_4]=e_1+e_2, [e_2,e_4]=e_2,$ \\
$[e_1,e_5]=\mu e_2, [e_3,e_5]=e_3$
\end{tabular}  & $3$  & $3$  \\
\hline
$L_{20}^{\mu}$ & \begin{tabular}{c}
$[e_1,e_4]=e_1+e_2, [e_2,e_4]=e_2,$ \\
$[e_1,e_5]=\mu e_2, [e_3,e_5]=e_3, [e_5,e_3]=-e_3$
\end{tabular}  & $3$  & $3$  \\
\hline
$L_{21}^{\mu}$ & \begin{tabular}{c}
$[e_1,e_4]=e_1+e_2, [e_2,e_4]=e_2,$ \\
$[e_1,e_5]=\mu e_2, [e_3,e_5]=e_3$, \\
$[e_4,e_1]=-e_1-e_2, [e_4,e_2]=-e_2, [e_5,e_1]=-\mu e_2$
\end{tabular}  & $3$  & $3$  \\
\hline
$L_{22}^{\mu}$ & \begin{tabular}{c}
$[e_1,e_4]=e_1+e_2, [e_2,e_4]=e_2,$ \\
$[e_1,e_5]=\mu e_2, [e_3,e_5]=e_3, [e_5,e_3]=-e_3$, \\
$[e_4,e_1]=-e_1-e_2, [e_4,e_2]=-e_2, [e_5,e_1]=-\mu e_2$
\end{tabular}  & $3$  & $3$  \\
\hline
\end{tabular}
\end{center}
\end{table}

\begin{table}[htp] \caption{$6$-dimensional nilpotent Leibniz algebras with derived algebra of codimension $2$ (I).}
\label{alphaybetaleibnizdim6nilpotentI}
\begin{center}
\begin{tabular}{|c|c|c|c|}
\hline
$L$ & Products & $\alpha(L)$ & $\beta(L)$ \\
\hline
$L_{1}$ & \begin{tabular}{c}
$[e_1,e_1]=e_6, [e_1,e_2]=e_3=-[e_2,e_1]$ \\
$[e_1,e_3]= e_4=-[e_3,e_1],[e_2,e_3]=e_5=-[e_3,e_2]$
\end{tabular}  & $4$  & $4$  \\
\hline
$L_{2}$ & \begin{tabular}{c}
$[e_1,e_1]=e_6, [e_1,e_2]=e_3=-[e_2,e_1]$ \\
$[e_2,e_2]=e_6, [e_1,e_3]= e_4=-[e_3,e_1],$ \\
$[e_2,e_3]=e_5=-[e_3,e_2]$
\end{tabular}  & $4$  & $4$  \\
\hline
$L_{3}$ & \begin{tabular}{c}
$[e_1,e_1]=e_6, [e_1,e_2]=e_3=-[e_2,e_1]$ \\
$[e_1,e_3]= e_4=-[e_3,e_1],[e_1,e_4]=e_5=-[e_4,e_1]$
\end{tabular}  & $5$  & $5$  \\
\hline
$L_{4}$ & \begin{tabular}{c}
$[e_1,e_1]=e_6, [e_1,e_2]=e_3=-[e_2,e_1]$ \\
$[e_1,e_3]= e_4=-[e_3,e_1],[e_2,e_3]=e_5=-[e_3,e_2]$ \\
$[e_1,e_4]= e_5=-[e_4,e_1],$
\end{tabular}  & $4$   & $4$   \\
\hline
$L_{5}$ & \begin{tabular}{c}
$[e_1,e_1]=e_6, [e_1,e_2]=e_3=-[e_2,e_1]$ \\
$[e_2,e_3]=e_4=-[e_3,e_2],[e_2,e_4]=e_5=-[e_4,e_2]$
\end{tabular}  & $4$  & $4$   \\
\hline
$L_{6}$ & \begin{tabular}{c}
$[e_1,e_1]=e_6, [e_1,e_2]=e_3=-[e_2,e_1]$ \\
$[e_1,e_3]=e_5=-[e_3,e_1],[e_2,e_3]=e_4=-[e_3,e_2]$ \\
$[e_2,e_4]=e_5=-[e_4,e_2]$
\end{tabular}  & $4$  &  $4$  \\
\hline
$L_{7}$ & \begin{tabular}{c}
$[e_1,e_1]=e_6, [e_1,e_2]=e_3=-[e_2,e_1]$ \\
$[e_2,e_2]=e_6, [e_1,e_3]= e_4=-[e_3,e_1],$ \\
$[e_1,e_4]=e_5=-[e_4,e_1]$
\end{tabular}  & $4$  & $4$   \\
\hline
$L_{8}$ & \begin{tabular}{c}
$[e_1,e_1]=e_6, [e_1,e_2]=e_3=-[e_2,e_1]$ \\
$[e_2,e_2]=e_6, [e_1,e_3]= e_4=-[e_3,e_1],$ \\
$[e_2,e_3]=e_5=-[e_3,e_2], [e_1,e_4]=e_5=-[e_4,e_1]$ \\
\end{tabular}  & $4$  & $4$   \\
\hline
$L_{9}$ & \begin{tabular}{c}
$[e_1,e_1]=e_6, [e_1,e_2]=e_3=-[e_2,e_1]$ \\
$[e_2,e_2]=e_6, [e_1,e_3]= e_4=-[e_3,e_1],$ \\
$[e_2,e_3]=ie_4=-[e_3,e_2], [e_1,e_4]=e_5=-[e_4,e_1]$ \\
$[e_2,e_4]=i e_5=-[e_4,e_2]$
\end{tabular}  & $5$  & $5 $  \\
\hline
$L_{10}$ & \begin{tabular}{c}
$[e_1,e_1]=e_6, [e_1,e_2]=e_3=-[e_2,e_1]$ \\
$[e_2,e_2]=e_6, [e_1,e_3]= e_4=-[e_3,e_1],$ \\
$[e_2,e_3]=ie_4+e_5=-[e_3,e_2],$ \\ $[e_1,e_4]=e_5=-[e_4,e_1],
[e_2,e_4]=i e_5=-[e_4,e_2]$
\end{tabular}  & $ 4$  & $ 4$  \\
\hline
\end{tabular}
\end{center}
\end{table}

\begin{table}[htp] \caption{$6$-dimensional nilpotent Leibniz algebras with derived algebra of codimension $2$ (II).}
\label{alphaybetaleibnizdim6nilpotentII}
\begin{center}
\begin{tabular}{|c|c|c|c|}
\hline
$L$ & Products & $\alpha(L)$ & $\beta(L)$ \\
\hline
$L_{11}$ & \begin{tabular}{c}
$[e_1,e_1]=e_6, [e_1,e_2]=e_3=-[e_2,e_1]$ \\
$[e_1,e_3]= e_4=-[e_3,e_1],[e_2,e_3]=e_5=-[e_3,e_2]$ \\ $[e_1,e_4]=e_6=-[e_4,e_1]$
\end{tabular}  & $4$  & $4$  \\
\hline
$L_{12}$ & \begin{tabular}{c}
$[e_1,e_1]=e_6, [e_1,e_2]=e_3=-[e_2,e_1]$ \\
$[e_1,e_3]= e_5=-[e_3,e_1],[e_2,e_3]=e_4=-[e_3,e_2]$ \\ $[e_2,e_4]=e_6=-[e_4,e_2]$
\end{tabular}  & $4$  & $4$  \\
\hline
$L_{13}$ & \begin{tabular}{c}
$[e_1,e_1]=e_6, [e_1,e_2]=e_3=-[e_2,e_1]$ \\
$[e_2,e_2]=e_6, [e_1,e_3]= e_4=-[e_3,e_1],$ \\ $[e_2,e_3]=e_5=-[e_3,e_2],[e_1,e_4]=e_6=-[e_4,e_1]$
\end{tabular}  & $ 4$  & $4 $  \\
\hline
$L_{14}$ & \begin{tabular}{c}
$[e_1,e_1]=e_6, [e_1,e_2]=e_3=-[e_2,e_1]$ \\
$[e_2,e_2]=e_6, [e_1,e_3]= e_4=-[e_3,e_1],$ \\ $[e_2,e_3]=ie_4+e_5=-[e_3,e_2],$ \\ $[e_1,e_4]=e_6=-[e_4,e_1], [e_2,e_4]=i e_6=-[e_4,e_2]$
\end{tabular}  & $4$  & $4$  \\
\hline
$L_{15}$ & \begin{tabular}{c}
$[e_1,e_1]=e_6, [e_1,e_2]=e_3=-[e_2,e_1]$ \\
$[e_1,e_3]= e_4=-[e_3,e_1],[e_2,e_3]=e_6=-[e_3,e_2]$ \\ $[e_1,e_4]=e_5=-[e_4,e_1]$
\end{tabular}  & $4 $  & $4 $  \\
\hline
$L_{16}$ & \begin{tabular}{c}
$[e_1,e_1]=e_6, [e_1,e_2]=e_3=-[e_2,e_1]$ \\
$[e_1,e_3]= e_4=-[e_3,e_1],[e_2,e_3]=e_5+e_6=-[e_3,e_2]$ \\ $[e_1,e_4]=e_5=-[e_4,e_1]$
\end{tabular}  & $ 4$  & $4 $  \\
\hline
$L_{17}$ & \begin{tabular}{c}
$[e_1,e_1]=e_6, [e_1,e_2]=e_3=-[e_2,e_1]$ \\
$[e_1,e_3]= e_6=-[e_3,e_1],$
$[e_2,e_3]=e_4=-[e_3,e_2]$, \\$[e_2,e_4]=e_5=-[e_4,e_2]$
\end{tabular}  & $4$  & $4$  \\
\hline
$L_{18}$ & \begin{tabular}{c}
$[e_1,e_1]=e_6, [e_1,e_2]=e_3=-[e_2,e_1]$ \\
$[e_1,e_3]= e_5+e_6=-[e_3,e_1],$ \\
$[e_2,e_3]=e_4=-[e_3,e_2], [e_2,e_4]=e_5=-[e_4,e_2]$ \\
\end{tabular}  & $ 4$  & $4 $  \\
\hline
$L_{19}^{a}$ & \begin{tabular}{c}
$[e_1,e_1]=e_6, [e_1,e_2]=e_3=-[e_2,e_1]$ \\
$[e_2,e_2]=e_6, [e_1,e_3]= e_4=-[e_3,e_1],$ \\
$[e_2,e_3]=a e_5+e_6=-[e_3,e_2],$ \\
$ [e_1,e_4]=e_5=-[e_4,e_1]$
\end{tabular}  & $ 4$  & $4 $  \\
\hline
$L_{20}$ & \begin{tabular}{c}
$[e_1,e_1]=e_6, [e_1,e_2]=e_3=-[e_2,e_1]$ \\
$[e_2,e_2]=e_6, [e_1,e_3]= e_4=-[e_3,e_1],$ \\
$[e_2,e_3]=ie_4+e_6=-[e_3,e_2],$ \\ $[e_1,e_4]=e_5=-[e_4,e_1],
[e_2,e_4]=i e_5=-[e_4,e_2]$
\end{tabular}  & $ 4$  & $4$  \\
\hline
\end{tabular}
\end{center}
\end{table}

\begin{table}[htp] \caption{$6$-dimensional nilpotent Leibniz algebras with derived algebra of codimension $2$ (III).}
\label{alphaybetaleibnizdim6nilpotentIII}
\begin{center}
\begin{tabular}{|c|c|c|c|}
\hline
$L$ & Products & $\alpha(L)$ & $\beta(L)$ \\
\hline
$L_{21}$ & \begin{tabular}{c}
$[e_1,e_1]=e_6, [e_1,e_2]=e_3=-[e_2,e_1]$ \\
$[e_2,e_2]=e_6, [e_1,e_3]=e_4=-[e_3,e_1],$ \\
$[e_2,e_3]=ie_4+e_5+e_6=-[e_3,e_2]$ \\ $[e_1,e_4]=e_5=-[e_4,e_1]$ \\
$[e_2,e_4]=i e_5=-[e_4,e_2]$
\end{tabular}  & $4$  & $ 4$  \\
\hline
$L_{22}$ & \begin{tabular}{c}
$[e_1,e_1]=e_6, [e_1,e_2]=e_3=-[e_2,e_1]$ \\
$[e_1,e_3]= e_4=-[e_3,e_1],[e_2,e_3]=e_5=-[e_3,e_2]$ \\ $[e_2,e_4]=e_6=-[e_4,e_2], [e_1,e_5]=e_6=-[e_5,e_1]$
\end{tabular}  & $ 4$  & $ 4$  \\
\hline
$L_{23}^{a}$ & \begin{tabular}{c}
$[e_1,e_1]=e_6, [e_1,e_2]=e_3=-[e_2,e_1]$ \\
$[e_1,e_3]= e_4=-[e_3,e_1],[e_2,e_3]=e_5=-[e_3,e_2]$  \\ $[e_1,e_4]=e_6=-[e_4,e_1], [e_2,e_4]=a e_6=-[e_4,e_2],$  \\
$[e_2,e_5]=e_6=-[e_5,e_2]$
\end{tabular}  & $ 4$  & $4 $  \\
\hline
$L_{24}$ & \begin{tabular}{c}
$[e_1,e_1]=e_6, [e_1,e_2]=e_3=-[e_2,e_1]$ \\
$[e_2,e_2]=e_6, [e_1,e_3]= e_4=-[e_3,e_1],$ \\ $[e_2,e_3]=e_5=-[e_3,e_2],$ \\ $[e_1,e_5]=e_6=-[e_5,e_1], [e_2,e_4]= e_6=-[e_4,e_2]$
\end{tabular}  & $ 4$  & $4 $  \\
\hline
$L_{25}^{a,b}$ & \begin{tabular}{c}
$[e_1,e_1]=e_6, [e_1,e_2]=e_3=-[e_2,e_1]$ \\
$[e_2,e_2]=e_6, [e_1,e_3]= e_4=-[e_3,e_1],$ \\ $[e_2,e_3]=e_5=-[e_3,e_2],[e_1,e_4]= a e_6=-[e_4,e_1],$ \\ $[e_1,e_5]=b e_6=-[e_5,e_1],[e_2,e_5]=e_6=-[e_5,e_2]$
\end{tabular}  & $ 4$  & $4$  \\
\hline
$L_{26}$ & \begin{tabular}{c}
$[e_1,e_1]=e_6, [e_1,e_2]=e_3=-[e_2,e_1]$ \\
$[e_1,e_3]= e_4=-[e_3,e_1],[e_1,e_4]= e_5=-[e_4,e_1],$ \\ $[e_1,e_5]=e_6=-[e_5,e_1], $
\end{tabular}  & $ 5$  & $5 $  \\
\hline
$L_{27}$ & \begin{tabular}{c}
$[e_1,e_1]=e_6, [e_1,e_2]=e_3=-[e_2,e_1]$ \\
$[e_1,e_3]= e_4=-[e_3,e_1],[e_2,e_3]=e_6=-[e_3,e_2]$ \\$[e_1,e_4]= e_5=-[e_4,e_1],[e_1,e_5]=e_6=-[e_5,e_1], $
\end{tabular}  & $ 4$  & $ 4$  \\
\hline
$L_{28}$ & \begin{tabular}{c}
$[e_1,e_1]=e_6, [e_1,e_2]=e_3=-[e_2,e_1]$ \\
$[e_1,e_3]= e_4=-[e_3,e_1],$ \\
$[e_2,e_3]=e_5=-[e_3,e_2], [e_1,e_4]=e_5=-[e_4,e_1]$ \\
$[e_2,e_4]=e_6=-[e_4,e_2], [e_1,e_5]=e_6=-[e_5,e_1]$ \\
\end{tabular}  & $4 $  & $ 4$  \\
\hline
$L_{29}$ & \begin{tabular}{c}
$[e_1,e_1]=e_6, [e_1,e_2]=e_3=-[e_2,e_1]$ \\
$[e_1,e_3]= e_4=-[e_3,e_1],[e_1,e_4]=e_5=-[e_4,e_1],$ \\ $[e_2,e_5]=e_6=-[e_5,e_2]$
\end{tabular}  & $ 4$  & $ 4$  \\
\hline
$L_{30}$ & \begin{tabular}{c}
$[e_1,e_1]=e_6, [e_1,e_2]=e_3=-[e_2,e_1]$ \\
$[e_1,e_3]= e_4=-[e_3,e_1],[e_2,e_3]=e_5=-[e_3,e_2]$ \\
$[e_1,e_4]=e_5=-[e_4,e_1]$,$[e_2,e_5]=e_6=-[e_5,e_2]$
\end{tabular}  & $ 4$  & $ 4$  \\
\hline
\end{tabular}
\end{center}
\end{table}

\begin{table}[htp] \caption{$6$-dimensional nilpotent Leibniz algebras with derived algebra of codimension $2$ (IV).}
\label{alphaybetaleibnizdim6nilpotentIV}
\begin{center}
\begin{tabular}{|c|c|c|c|}
\hline
$L$ & Products & $\alpha(L)$ & $\beta(L)$ \\
\hline
$L_{31}$ & \begin{tabular}{c}
$[e_1,e_1]=e_6, [e_1,e_2]=e_3=-[e_2,e_1]$ \\
$[e_2,e_3]= e_4=-[e_3,e_2],[e_2,e_4]=e_5=-[e_4,e_2]$ \\
$[e_2,e_5]=e_6=-[e_5,e_2]$
\end{tabular}  & $ 4$  & $ 4$  \\
\hline
$L_{32}$ & \begin{tabular}{c}
$[e_1,e_1]=e_6, [e_1,e_2]=e_3=-[e_2,e_1]$ \\
$[e_1,e_3]= e_6=-[e_3,e_1],[e_2,e_3]=e_4=-[e_3,e_2]$\\
$[e_2,e_4]=e_5=-[e_4,e_2]$, $[e_2,e_5]=e_6=-[e_5,e_2]$
\end{tabular}  & $ 4$  & $ 4$  \\
\hline
$L_{33}$ & \begin{tabular}{c}
$[e_1,e_1]=e_6, [e_1,e_2]=e_3=-[e_2,e_1]$ \\
$[e_1,e_3]= e_5=-[e_3,e_1],[e_2,e_3]=e_4=-[e_3,e_2]$\\
$[e_1,e_4]=e_6=-[e_4,e_1]$, $[e_2,e_4]=e_5=-[e_4,e_2]$ \\
$[e_2,e_5]=e_6=-[e_5,e_2]$
\end{tabular}  & $ 4$  & $4 $  \\
\hline
$L_{34}$ & \begin{tabular}{c}
$[e_1,e_1]=e_6, [e_1,e_2]=e_3=-[e_2,e_1]$ \\
$[e_2,e_3]= e_4=-[e_3,e_2],[e_2,e_4]=e_5=-[e_4,e_2]$\\
$[e_1,e_5]=e_6=-[e_5,e_1]$,
\end{tabular}  & $ 4$  & $ 4$  \\
\hline
$L_{35}$ & \begin{tabular}{c}
$[e_1,e_1]=e_6, [e_1,e_2]=e_3=-[e_2,e_1]$ \\
$[e_1,e_4]=e_5=-[e_4,e_1], [e_2,e_3]= e_4=-[e_3,e_2],$ \\
$[e_2,e_4]=e_5=-[e_4,e_2]$, $[e_1,e_5]=e_6=-[e_5,e_1]$
\end{tabular}  & $ 4$  & $ 4$  \\
\hline
$L_{36}^{a}$ & \begin{tabular}{c}
$[e_1,e_1]=e_6, [e_1,e_2]=e_3=-[e_2,e_1]$ \\
$[e_1,e_3]=ae_5=-[e_3,e_1], [e_2,e_3]= e_4=-[e_3,e_2],$ \\
$[e_2,e_4]=e_5=-[e_4,e_2]$, $[e_1,e_5]=e_6=-[e_5,e_1]$ \\
$[e_2,e_5]=e_6=-[e_5,e_2]$
\end{tabular}  & $ 4$  & $ 4$  \\
\hline
$L_{37}^{a}$ & \begin{tabular}{c}
$[e_1,e_1]=e_6, [e_1,e_2]=e_3=-[e_2,e_1]$ \\
$[e_2,e_2]=e_6, [e_1,e_3]=e_4=-[e_3,e_1],$ \\ $[e_2,e_3]= ae_6=-[e_3,e_2],$ \\
$[e_1,e_4]=e_5=-[e_4,e_1]$, $[e_1,e_5]=e_6=-[e_5,e_1]$
\end{tabular}  & $ 4$  & $4 $  \\
\hline
$L_{38}^{a,b}$ & \begin{tabular}{c}
$[e_1,e_1]=e_6, [e_1,e_2]=e_3=-[e_2,e_1]$ \\
$[e_2,e_2]=a e_6, [e_1,e_3]=e_4=-[e_3,e_1],$ \\ $[e_2,e_3]= e_5+be_6=-[e_3,e_2],$ \\
$[e_1,e_4]=e_5=-[e_4,e_1]$,
$[e_2,e_4]=e_6=-[e_4,e_2]$, \\
$[e_1,e_5]=e_6=-[e_5,e_1]$
\end{tabular}  & $ 4$  & $ 4$  \\
\hline
\end{tabular}
\end{center}
\end{table}

\begin{table}[htp] \caption{$6$-dimensional nilpotent Leibniz algebras with derived algebra of codimension $2$ (V).}
\label{alphaybetaleibnizdim6nilpotentV}
\begin{center}
\begin{tabular}{|c|c|c|c|}
\hline
$L$ & Products & $\alpha(L)$ & $\beta(L)$ \\
\hline
$L_{39}^{a}$ & \begin{tabular}{c}
$[e_1,e_1]=e_6, [e_1,e_2]=e_3=-[e_2,e_1]$ \\
$[e_2,e_2]=e_6, [e_1,e_3]=e_4=-[e_3,e_1],$ \\
$[e_1,e_4]=e_5=-[e_4,e_1]$,
$[e_1,e_5]=ae_6=-[e_5,e_1]$, \\
$[e_2,e_5]=e_6=-[e_5,e_2]$
\end{tabular}  & $ 4$  & $ 4$  \\
\hline
$L_{40}^{a,b}$ & \begin{tabular}{c}
$[e_1,e_1]=e_6, [e_1,e_2]=e_3=-[e_2,e_1]$ \\
$[e_2,e_2]=a e_6, [e_1,e_3]=e_4=-[e_3,e_1],$ \\
$[e_2,e_3]=e_5=-[e_3,e_2]$
$[e_1,e_4]=e_5=-[e_4,e_1]$, \\
$[e_2,e_4]=be_6=-[e_4,e_2]$, \\
$[e_1,e_5]=be_6=-[e_5,e_1]$,
$[e_2,e_5]=e_6=-[e_5,e_2]$
\end{tabular}  & $ 4$  & $ 4$  \\
\hline
$L_{41}^{a,b}$ & \begin{tabular}{c}
$[e_1,e_1]=e_6, [e_1,e_2]=e_3=-[e_2,e_1]$ \\
$[e_2,e_2]= e_6, [e_1,e_3]=e_4=-[e_3,e_1],$ \\
$[e_2,e_3]=ie_4+ae_5=-[e_3,e_2]$ \\
$[e_1,e_4]=e_5=-[e_4,e_1]$, \\
$[e_2,e_4]=ie_5+be_6=-[e_4,e_2]$, \\
$[e_1,e_5]=e_6=-[e_5,e_1]$,
\end{tabular}  & $ 4$  & $4 $  \\
\hline
$L_{42}^{a,b}$ & \begin{tabular}{c}
$[e_1,e_1]=e_6, [e_1,e_2]=e_3=-[e_2,e_1]$ \\
$[e_2,e_2]=e_6, [e_1,e_3]=e_4=-[e_3,e_1],$ \\
$[e_2,e_3]=ie_4+ae_5=-[e_3,e_2]$ \\
$[e_1,e_4]=e_5=-[e_4,e_1]$, \\
$[e_2,e_4]=ie_5+be_6=-[e_4,e_2]$, \\
$[e_1,e_5]=e_6=-[e_5,e_1]$,
$[e_2,e_5]=ie_6=-[e_5,e_2]$
\end{tabular}  & $ 4$  & $4 $  \\
\hline
\end{tabular}
\end{center}
\end{table}

Note that Leibniz algebra $L_3$ of Table \ref{alphaybetaleibnizdim5_nilrad3_I} is an example over the complex numbers for which $\alpha$ and $\beta$ are different, unlike the Lie algebra case (see \cite[Proposition 2.6]{bc}).

\end{document}